\documentclass[[amscd,amssymb,verbatim,12pt]{amsart}
\usepackage{amssymb,latexsym, amsmath, amscd, array, graphicx, relsize}
\RequirePackage[OT1]{fontenc}
\RequirePackage{amsthm,amsmath,amssymb}
\RequirePackage[numbers,square]{natbib} %
\RequirePackage[colorlinks]{hyperref}
\usepackage[all]{xy}
\usepackage{latexsym, amsmath, amscd, array, graphicx,relsize}
\xyoption{v2}

\swapnumbers \numberwithin{equation}{section}

\theoremstyle{break}

\newtheorem{thm}{Theorem}[section]

\newtheorem{lem}[thm]{Lemma}

\newtheorem{prop}[thm]{Proposition}
\newtheorem{cor}[thm]{Corollary}

\newtheoremstyle{break}
   {9pt}
   {9pt}
{\rmfamily}
   {}
   {\bfseries}
   {.}
   {\newline}
{}

\theoremstyle{definition}
\newtheorem{defn}[thm]{Definition}

\newtheorem{rem}[thm]{Remark}

\def\Z{{\mathbb Z}}

\def\1{\hbox{\rm\rlap {1}\hskip.03in{\rom I}}}
\def\Bbbone{{\rm1\mathchoice{\kern-0.25em}{\kern-0.25em}
{\kern-0.2em}{\kern-0.2em}I}}

\long\def\forget#1\forgotten{} %

\newcommand\ver[1]{\marginpar{\tiny Changed in Ver \VER}}

\newcounter{notecounter}

\date{\today}

\title[cell-like maps and surgery]{Cell-like maps and surgery}
\author{Michelle Daher}

\begin{document}

\begin{abstract}
We show that for any $n\geq6$, for any $m$, we can find $m$ non-homeomorphic $n$-manifolds that can be mapped by cell-like maps onto the same space $X$.
\end{abstract}

\maketitle

\section {introduction}
Cell-like maps constitute a natural useful class of maps for many reasons. Suppose that a sequence of homeomorphisms $h_n:M\to N$ of closed manifolds converges to a continuous map $h:M\to N$,
then $h$ is not necessarily a homeomorphism. It is the next best thing - a cell-like map. By definition a proper map $f:X\to Y$ is called {\em cell-like} if it has cell-like point preimages $f^{-1}(y)$ for all $y$. When $\dim X<\infty$ this condition means that $f^{-1}(y)$ can be presented as the intersection of a nested sequence of closed cells in some euclidian space. Since the empty set is not cell-like, every cell-like map is surjective.

Cell-like maps of a manifold $M$ can be constructed by means of upper semi-continuous decomposition of $M$ into cell-like sets. Many interesting cell-like maps of lower dimensional manifolds were constructed that way by Bing and his school~\cite{B}. If $q:M\to X$ is the quotient map for such a decomposition and if $X$ is a manifold, then $X$ is homeomorphic to $M$ and $q$ can be approximated by homeomorphisms. This is due to the Siebenmann Approximation Theorem~\cite{Edwards}.
In the case where $X$ is not a manifold, Lacher, in 1977, asked the following question: {\em Can one find two nonhomeomorphic closed manifolds $M_1$ and $M_2$ that map onto the same $X$ by cell-like maps ?}

The image $X$ of a cell-like map $f:M\to X$ of a manifold still carries manifold features, in particular, it is a homology manifold. A typical example of such is a manifold with certain singularities. Hence, it's quite natural to call the domain $M$ of a cell-like map $f:M\to X$  a {\em resolution} of $X$. In the late 70s, Quinn proved the uniqueness of resolution theorem~\cite{Q}, Proposition 3.2.3,  which implies that if $\dim X<\infty$, then the answer to Lacher's question is negative.

Thus, a positive answer to Lacher's question is possible only for $\dim X=\infty$.  Whether cell-like maps can raise the dimension to infinity was a big open problem in the area since the 50s. In the late 70s, R. Edwards proved that this problem is equivalent to the even older (early 30s) Alexandroff problem of {\em whether
the covering dimension coincides with the cohomological dimension for compact metric spaces}. In the 80s, Dranishnikov settled this problem \cite{Wa}.
In 2006, Dranishnikov and Ferry annouced an example of two closed non-homeomorphic 7-dimensional manifolds that can be mapped by cell-like maps onto the same space~\cite{Dr}. In 2020~\cite{Dr1},  Dranishnikov, Ferry, and Weinberger gave an example, Corollary 2.15, in dimension 6.

In this paper we answer the following version of Lacher's question: {\em How many pairwise nonhomeomorphic manifolds can be  mapped by cell-like maps onto the same space?}
Ferry proved ~\cite{F2} that we cannot have infinitely many such maniolds.
 In this paper we show that:
 \begin{thm}
 For any $n\geq6$, for any $m$, there are $m$ closed non-homeomorphic $n$-manifolds that can be mapped by cell-like maps onto the same space $X$.
 \end{thm}
 \noindent This result is based on the work in \cite{Dr1}.
 \\       
\\\indent We recall that the topological structure group $S(M)$ on a manifold $M$ consists of classes of (simple) homotopy equivalences $f:N\to M$. In~\cite{Dr1},  the authors studied the subgroup $S^{CE}(M)$ of $S(M)$ which is defined by homotopy equivalences that factor through cell-like maps. This means that for every  $[f:N\to M]\in S^{CE}(M)$, there is $X$ and and cell-like maps $q_1:N\to X$ and $q_2:M\to X$ such that the diagram  
 \[
\xymatrix{ N \ar[rr]^{f}\ar[dr]_{q_1}& &M\ar[dl]^{q_2}\\
	& X &
}
\] 
\noindent homotopy commutes.
\\
\\\indent In this paper,  we show that often (but not always)
a common space $X$ can be chosen for all elements of $S^{CE}(M)$.
        
\begin{thm}Let $M$ be an $n$-dimensional closed connected topological manifold, $n\ge6$ and $G$ be a finitely generated subgroup of $S^{CE}(M)$, then there exists a space $X$ such that for any $[N,f]\in G$, $f:N\to M$ factors through cell-like maps to $X$.
 \end{thm}       
 \noindent In particular, when $M$ is simply connected, we can find one $X$ that works for all the classes in $S^{CE}(M)$:

\begin{prop} \cite{Dr1} If $L_{n+1}(\mathbb{Z}\pi_1(M))$ has finitely generated odd torsion, then $S^{CE}(M)$ is finite.
\end{prop}

Also, we note that there are closed manifolds with infinitely generated $S^{CE}(M)$~\cite{Dr1}.  
\\
\\ The author would like to thank Alexander Dranishnikov on his input and guidance throughout this project.

\section{preliminaries}

All manifolds in this paper are closed, orientable, connected, topological manifolds of dimension $\ge 6$. Since ordinary homology does not behave well for non-ANRs, we will use the Steenrod extension of a generalized homology theory $H_*$,\cite{Milnor1},\cite{F1}, \cite{CP}, \cite{EH}, which satisfies the usual Eilenberg-Steenrod axioms
for (generalized) homology theories, together with the union axiom. We denote by $\widetilde H_*$ its reduced version and we define it for pairs by setting $\widetilde H_*(X,Y)=\widetilde H_*(X/Y)=H_*(X,Y)$. Recall the axioms of a reduced Steenrod homology theory:
\\
\\(a) Exactness: For a compact metrizable pair $(X,Y)$ there's a long exact sequence 
$$\cdot\cdot\cdot\to \widetilde H_i(Y) \to \widetilde H_i(X) \to \widetilde H_i(X/Y) \to \widetilde H_{i-1}(Y)\to\cdot\cdot\cdot$$
\\ (b) Milnor's additivity axiom:  Given a countable collection $X_i$ of pointed compact metric spaces and letting
$\bigvee X_i \subset \prod X_i$ be the null wedge, we have an isomorphism 
$$\widetilde H_*(\bigvee X_i) \cong \prod \widetilde H_*(X_i).$$
\\
\\We need the following results \cite{AH},\cite{BM},\cite{Y}. Denote by $M(p)$ the $\mathbb Z_p$ Moore spectrum. 
\begin{thm} If $p>1$ is an integer and $n \ge 3$, $\widetilde{H_*}(K(\pi,n);KO\wedge\Z_p)=0$
for any group $\pi$. If $\pi$ is torsion,  $\widetilde{H_*}(K(\pi,n);KO)=0$, for $n= 2$.
\end{thm}
\noindent For odd $p$, we have a chain of homotopy equivalences of spectra:
$$
\mathrm{KO}\wedge M(p)\sim \widetilde{KO}[\frac{1}{2}]\wedge M(p)\sim\mathbb L[\frac{1}{2}]\wedge M(p)\sim
\mathbb L\wedge M(p).
$$
\\ \noindent We recall the Universal Coefficient Formula
for coefficients  in $\Z_p$ for an extraordinary homology theory given by a spectrum $\mathbb E$
of CW complexes.
$$0\to H_n(K;\mathbb E)\otimes\Z_p\to H_n(K;\mathbb E\wedge M(p))\to
\mathrm{Tor}(H_{n-1}(K;\mathbb E),\,\Z_p)\to 0$$
where $\mathrm{Tor}(H,\Z_p)=\{c\in H|pc=0\}$.

\begin{defn}
Given a CW-complex $K$, denote by $P_2(K)$ the 
\\$second\  stage\  of\  the\  Postnikov\  tower\  of\  K$. It is the CW-complex obtained by attaching to $K$ cells in dimensions 4 and higher to kill the homotopy groups of $K$ in dimensions 3 and above. Thus, $P_2(K)-K$ consists of cells of dimension $\geq4$ and $\pi_n(P_2(K))=0$ for $n\geq3$.
\end{defn}

\begin{defn}
A compact subset $X$ of an ANR $M$ is $\emph{cell-like}$ if for each neighborhood $U$ of $X$ the inclusion $i:X\hookrightarrow U$ is nullhomotopic.
\end{defn}
Cell-likeness is an intrinsic property of the space $X$. It can be easily shown that if $X$ is a compact ANR, then being cell-like is equivalent to being contractible.
A classical example of a space that is cell-like but not contractible is the topologist's sine curve. For expositions on cell-like maps, see \cite{Edwards}, \cite{La1}, \cite{La2}.

\begin{defn}
A map $f:X\to Y$ between compact metric spaces is cell-like provided that each point inverse is cell-like. 
\end{defn}

It follows from the definition of cell-like spaces that a cell-like map is surjective.

\begin{defn}
A homotopy equivalence $f:N\to M$ between closed manifolds $\emph{factors through cell-like maps}$ if there is a space $X$ and cell-like maps $c_1:N\to X$ and $c_2:M\to X$ such that the following diagram
 \[
    \xymatrix{ N \ar[rr]^{f}\ar[dr]_{c_1}& &M\ar[dl]^{c_2}\\
    & X &
    }
    \]
    
\noindent homotopy commutes. We say that $N$ and $M$ are $CE-related$.

\end{defn}

\begin{rem} Note that if $c_1:N\to X$ and $c_2:M\to X$ are cell-like maps, then there exists a homotopy equivalence $f:N\to M$ so that $c_2\circ f\sim c_1$. This follows from \cite{La2} Lemma 2.3.
\end{rem}

We recall the definition of the simple topological structure set of a manifold $M$ and the algebraic surgery exact sequence. 
\begin{defn}

A $\emph{simple topological structure}$ $(N,f)$ on an $n$-dimensional manifold $M$ is an $n$-dimensional manifold $N$ together with a simple homotopy equivalence $f:N\to M$.

\end{defn}

\begin{defn}
The $\emph{simple topological structure set}$, for which we will omit decorations and denote here by $S(M)$, of an $n$-dimensional manifold $M$ is the set of equivalence classes of simple manifold structures on $M$. That is, $(N,f)\sim (N',f')$ if there is a homeomorphism $h$ so that the diagram 
 \[
    \xymatrix{ N \ar[rr]^{h}\ar[dr]_{f}& &N'\ar[dl]^{f'}\\
    & M &
    }
    \]
homotopy commutes.
\end{defn}

Denote by $S^{CE}(M)$ the set of topological structures that factor through cell-like maps. All homotopy equivalences in $S^{CE}(M)$ are simple \cite{Dr1}, Proposition 2.3.
\noindent The structure set of a closed topological $n$-manifold, $n\ge5$ fits into the Sullivan-Wall surgery exact sequence \cite{Wall} 

$$\cdots  L_{n+1}(\mathbb{Z}\pi_{1}(M)) \to S(M) \to [M,G/TOP] \stackrel{\theta}\to
L_{n}(\mathbb{Z}\pi_{1}(M)).$$

In this paper, weconsider the algebraic surgery exact sequence \cite{Dr1}:
$$\xymatrix@C.15in{
\cdots \ar[r] & L_{n+1}(\mathbb{Z}\pi_{1}(M))\ar[r] & S_n(M) \ar[r]^{\eta'}& H_n(M;\mathbb{L})\ar[r]^{A'} \ar[r] & L_{n}(\mathbb{Z}\pi_{1}(M))\ar[r] &\cdots}$$

\noindent where $H_n(M;\mathbb{L})=H^0(M;\mathbb{L})=[M,G/TOP\times\mathbb{Z}]$, $S(M)\subset S_n(M)$. 
\\
\\ \noindent For $M$ closed and connected, there's a split monomorphism \cite{Dr1}
$$\xymatrix {0\ar[r] & S(M)\ar[r]^{i}& S_n(M)\ar[r]& \mathbb{Z}}.$$

\noindent So, $S_n(M)$ differs from $S(M)$ by at most a $\mathbb{Z}$.

\begin{prop}\cite{Wa}Let $E$ be a CW complex with trivial homotopy groups $\pi_i(E)=0$,
$i\ge k$ for some $k$, and $q:X\to Y$ a cell-like map
between compacta. Then $q$ induces a bijection of the homotopy
classes $q^*:[Y,E]\to[X,E]$.
\end{prop}
Let $q:M\to X$ be a cell-like map and $j:M\to P_2(M)$ be the inclusion map. Then by the above proposition, there is a map $g:X\to P_2(M)$ such that $g\circ q\sim j$. Denote by $i$ the induced map on their mapping cylinders, $i:M_q\to M_j$ where $i|_M=\mathrm{id}|_M$ and $i|_X=g$, and by $i_*:H_*(M_q,M;\mathbb L)\to H_*(P_2(M),M;\mathbb L)$ the induced homomorphism
on the Steenrod \( \mathbb L \)-homology groups \cite{F1}\cite{KS}.
\\
\\
\noindent We state the main results of \cite{Dr1}.
\\
\\ \noindent Let $M$ be a closed $n$-manifold, there's a map
$$\delta :H_{n+1}(P_2(M),M;\mathbb{L}) \cong S_{n+1}( P_2(M),M)\stackrel{\partial}\to
S_n(M)\stackrel{p}\to S(M).
$$
where $p$ is any splitting of $i$. For details on the maps above, refer to \cite{Dr1}, section 2.

\begin{thm} (\cite{Dr1}, Theorems 2.4 and 2.7) Let $M^n$ be a closed topological manifold, $n\geq6$. Let $T_{odd}(H_{n+1}(P_2(M),M;\mathbb{L}))$ be the odd torsion subgroup of $H_{n+1}(P_2(M),M;\mathbb{L})$. \\ Then $S^{CE}(M)=\delta (T_{odd}(H_{n+1}(P_2(M),M;\mathbb{L})))$. In particular, $S^{CE}(M)$ is a subgroup of the odd torsion subgroup of $S(M)$. 
\\ Moreover, if $M$ is simply connected with finite $\pi_2(M)$, then $S^{CE}(M)$ is the odd torsion subgroup of $S(M)$.
\end{thm}

\section {Main Results}   

In \cite{Dr1} Corollary 2.15 the authors constructed a closed simply connected 6-dimensional manifold $M$ with $H_6(M;\mathbb{L})=\Z_2\oplus\Z_p\oplus\Z$ and showed, using the Browder-Novikov-Wall surgery exact sequence and a bundle theoretic argument, that for any element $[N,f]\in S^{CE}(M)$, $N$ is not homeomorphic to $M$ (We note that in~\cite{Dr1}, in the computation of $H_6(M;\mathbb{L})$, the 
$\mathbb Z$ and $\Z_2$ summands were mistakenly omitted. However, that does not change the final result since $S^{CE}(M)=\Z_p$ regardless). The manifold $M$ was constructed by PL-embedding the Moore complex $P=S^1\cup_pB^2$ in $\mathbb{R}^6$, then suspending to obtain $P'=S^2\cup_p B^3$ and taking $M=\partial W$, where $W$ is the regular neighborhood of $P'=S^2\cup_p B^3$ in $\mathbb{R}^7$.
  We state this result for completeness.
\begin{thm}[\cite{Dr1} Corollary 2.15] There are closed non-homeomorphic 6-dimensional manifolds $M$ and $N$ which are CE-related.
\end{thm}

In the next corollary, we generalize the authors' construction to any $n$-dimensional, $n\geq6$, manifolds $M$ and $N$. First, we need the following fact.
\begin{prop}(\cite{Ra}, Proposition 20.3)
For $n\geq4$ the structure set $S(M)$ of a simply connected $n$-dimensional topological manifold $M$ is such that $S(M)=S_{n+1}(M)=\mathrm{ker}(A:H_n(M;\mathbb{L}_\cdot)\to L_n(\mathbb{Z}))$ if $n\equiv0\ \mathrm{mod}\ 2$, and $S(M)=S_{n+1}(M)=H_n(M;\mathbb{L}_\cdot)$ if $n\equiv1\ \mathrm{mod}\ 2$.
\end{prop}

\begin{rem}Note that in this paper we use $H_n(M;\mathbb{L})$ and $S_n(M)$ which differ from $H_n(M;\mathbb{L}_\cdot)$ and $S_{n+1}(M)$ by at most a $\mathbb{Z}$ (See preliminaries and \cite{Dr1}, Proposition 4.5). Also, recall that $L_n(\mathbb{Z})=\mathbb{Z}$ if $n=4k$, $L_n(\mathbb{Z})=\mathbb{Z}_2$ if $n=4k+2$, and $L_n(\mathbb{Z})=0$ for $n$ odd.
\end{rem}

\begin{cor}
For any $n\geq6$, for any $m$, there are $m$ closed non-homeomorphic simply connected $n$-manifolds $M$ and $N$ which are CE-related.
\end{cor}
\begin{proof}
  Let $p$ be a prime number such that $p\geq m$ and $p$ does not divide $\mathrm{order}(x_i)$ for every $x_i\in \pi_i(BG), i\leq n+1$. This is possible since the groups $\pi_*(BG)$ are finite. Let $n\neq 7, 8$. By general position, the Moore complex
$P=S^1\cup_pB^2$ can be PL-embedded in $\mathbb{R}^6$. Suspending $n-5$ times 
embeds $P'=S^{n-4}\cup_{p}B^{n-3}$ in $\mathbb{R}^{n+1}$. Let $W$ be a regular neighborhood 
of $P'$ in $\mathbb{R}^{n+1}$ and let 
$\partial W = M$. The manifold $M$ is stably parallelizable because it is a closed codimension one submanifold of euclidean space $\mathbb{R}^{n+1}$ and it is simply connected.
By Lefschetz duality, $H_{n-4}(W,M)=H^5(W)=H^5(P')=0$ and
$H_{n-3}(W,M)=H^4(W)=H^4(P')=0$. Hence, the exact sequence of the pair
$(W,M)$ implies that $H_{n-4}(M)=\mathbb{Z}_p$. Similarly, we can find that $H_3(M)=\Z_p$ and $H_k(M)=0$ for $k\neq 0, 3, n-4, n$.
By the Atiyah-Hirzebruch spectral sequence $H_{n-4}(M;\mathbb{L})\cong L_n(\Z)\oplus H_3(M,L_{n-7}(\Z))\oplus \mathbb{Z}_p\oplus\mathbb Z $. Therefore, $H_{n}(M;\mathbb{L})\cong L_n(\Z)\oplus H_3(M,L_{n-7}(\Z))\oplus \mathbb{Z}_p\oplus\mathbb Z$. For $n=7$, from the exact sequence of pairs $(W,M)$, we get $H_{3}(M)=\mathbb{Z}_p\oplus\Z_p$ and $H_k(M)=0$ for $n\neq 0,3,7$. By the Atiyah-Hirzebruch spectral sequence $H_{3}(M;\mathbb{L})\cong \mathbb{Z}_p\oplus\Z_p\oplus\mathbb Z $. Hence, $H_{7}(M;\mathbb{L})\cong\Z_p\oplus\Z_p\oplus \mathbb{Z}$. For $n=8$, consider the PL-embedding of the Moore complex $P=S^1\cup_pB^2$ in $\mathbb{R}^7$ and suspend twice to embed $P'=S^{3}\cup_{p}B^{4}$ in $\mathbb{R}^{9}$. As before, let $M=\partial W$, where $W$ is the regular neighborhood of $P'$ in $\mathbb{R}^{9}$. By a similar argument, we get $H_3(M)=H_4(M)=\Z_p$, $H_k(M)=0$ for $k\neq 0,3,4,8$ and $H_{4}(M;\mathbb{L})\cong H_{8}(M;\mathbb{L}) \cong\mathbb Z\oplus \mathbb{Z}_p\oplus\mathbb Z $.
Hence, by Proposition 3.2, Remark 3.3, and Theorem 2.11, $S^{CE}(M)=\Z_p\oplus\Z_p$ for $n\equiv 3\ \mathrm{mod}\ 4$ and $S^{CE}(M)=\mathbb{Z}_p$ otherwise.

We get the following commutative diagram from the  Sullivan-Wall and the Quinn-Ranicki exact sequences. 
$$
\xymatrix{
L_{n+1}(\mathbb{Z})\ar[r]\ar[d]^{=} &S(M) \ar[r]^(.4){\eta}\ar[d]^{\subset}& [M;\,G/TOP]\ar[d]\ar[r] &L_n(\mathbb{Z})\ar[d]^{=}\\
L_{n+1}(\mathbb{Z})\ar[r]& S_n(M) \ar[r]^(.4){\eta'} & H_n(M,\mathbb{L}) \ar[r] &L_n(\mathbb{Z})
}
$$

Choose a nontrivial $p$-torsion element $\beta=[N,f]\in S^{CE}(M)$.
We next show that $N$ is not homeomorphic to $M$ by showing that $N$ has a nontrivial topological stable normal bundle.
Let $[\gamma]=\eta(\beta)$.
The class $[\gamma]$ represents the difference between topological stable normal bundles on $M$ and $N$ which are defined by two lifts
$\nu_M:M\to BTOP$ and $\sigma:M\to BTOP$  of the Spivak map $M\to BG$ with respect to the fibration $j:BTOP\to BG$. 
Here $\nu_M$ denotes a classifying map for the topological stable normal bundle on $M$.
Note that $\nu_N=\sigma\circ f$. Thus, the lifts $\nu_M$ and $\sigma$ are not fiberwise homotopic.
We need to show that $\nu_M$ and $\sigma$ are not homotopic in $BTOP$.

Since the stable normal bundle of $M$ is trivial, the map $\nu_M: M\to BTOP$ is nullhomotopic.
Note that the map $\sigma$ is homotopic to $i\circ\gamma$ where
$i:G/TOP\to BTOP$ is the inclusion of the fiber into the total space of the fibration $j$. The homotopy exact sequence of the fibration $j$ implies that after inverting by finitely many primes $p_1, p_2, ..., p_k$, where $p_i\neq p$, we get that the inclusion $i$  is an $(n+1)$-equivalence. In particular, $i_*: [M,G/TOP]\to [M,BTOP]$ is a bijection. Therefore, the map $i\circ\gamma$ is not nullhomotopic.

 Thus, $\nu_N$ is not nullhomotopic, the topological stable normal bundle of $N$ is nontrivial and, hence, $N$ is not homeomorphic to $M$.
\end{proof}

\begin{defn}
We define the reduced group $\widetilde S^{CE}(M)$ of $S^{CE}(M)$ to be the quotient group $S^{CE}(M)/\langle[N,f] -[N,\psi\circ f]\rangle$ where $f:N\to M$ and  $\psi : M\to M$ are orientation preserving homotopy equivalences that factor through cell-like maps.
\end{defn}
\noindent The above definition makes sense since $h_1\circ h_2$ factors through cell-like maps whenever $h_1$ and $h_2$ factor through cell-like maps \cite{Dr}, Corollary 2.11. Note that two different classes in $\widetilde S^{CE}(M)$ contain nonhomemorphic manifolds, for let $[N,f]$ and $[N,g]$ be elements of $S^{CE}(M)$ then $g=\psi\circ f$, where $\psi=gf^{-1}, \psi:M\to M$.

\begin{cor} Let $M$ be as in Corollary 3.4. Any two classes [N,f] and [P,g] in $S^{CE}(M)$ are such that $N$ and $P$ are non-homeomorphic.
\end{cor}
\begin{proof}
From Corollary 3.4, we have $S^{CE}(M)=\Z_p\oplus\Z_p$ for $n\equiv 3\ \mathrm{mod}\ 4$ and $S^{CE}(M)=\mathbb{Z}_p$ otherwise and nontrivial classes of  $S^{CE}(M)$ are classes of manifolds not homeomorphic to $M$. Hence, for any self homotopy equivalence $\psi$ of $M$, $[M,\psi]=[M,id]$ and therefore $\psi$ is homotopic to a homeomorphism. Therefore, $[N,f]=[N,\psi\circ f]$ and it follows that $S^{CE}(M) = \widetilde S^{CE}(M)$.
\end{proof}

Next, we generalize a result due to Dranishnikov (\cite{Dr2}, Theorem 7.2) which we need for the proof of Proposition 4.1. Dranishnikov showed that, under some assumptions, for a compact polyhedral pair $(P,L)$ and a non-zero element $\alpha\in\widetilde{H}_*(P,L)$, there is a compact metric space $Y\supset L$ and a map $f:(Y,L)\to (P,L)$ such that $\alpha\in \mathrm{Im}(f_*)$.

\begin{defn}
The cohomological dimension of a compactum $X$ with coefficients in the group $G$, denoted $\mathrm{dim}_GX$, is the largest integer $n$ such that $H^n(X,A;G)\neq0$ for some closed subset $A$ of $X$.
\end{defn}
\smallskip

\begin{defn}
A map $f:(X,L)\to (Y,L)$ is called strict if $f(X-L)=Y-L$ and $f|_L=\mathrm{id}_L$.
\end{defn}
\smallskip

\begin{thm}\cite{Dr2}(Theorem 7.2)Let $H_*$ be a generalized homology theory. Suppose that $\widetilde H_*(K(G,n))=0$. Then for any finite polyhedron pair $(K,L)$ and any element $\alpha\in H_*(K,L)$ there is a compactum $Y\supset L$ and a strict map $f:(Y,L)\to (K,L)$ such that
\\(a) $\mathrm{dim}_G(Y-L)\leq n$
\\(b) $\alpha\in \mathrm{Im}(f_*)$
\end{thm}
\smallskip
\begin{lem}Let $H_*$ a generalized homology theory. Suppose that $\widetilde H_*(K(G,n))=0$. Then for any finite polyhedron pair $(K,L)$ and $\alpha_1,...,\alpha_k\in H_*(K,L)$ there is a compactum $Y\supset L$ and a strict map $f:(Y,L)\to (K,L)$ such that
\\(a) $\mathrm{dim}_G(Y-L)\leq n$
\\(b) $\alpha_1,...,\alpha_k\in \mathrm{Im}(f_*)$
\end{lem}
\begin{proof}
Let $\alpha_1,...,\alpha_k\in H_*(K,L)$. By the above theorem, there are compacta $Y_i\supset L$ and strict maps $f_i:(Y_i,L)\to (K,L)$ such that $\mathrm{dim}_G(Y_i-L)\leq n$ and $\alpha_i\in \mathrm{Im}(f_{i_*})$, $i\in\{1,...,k\}$.
\\Let $Y=\coprod_L Y_i$, the spaces $Y_i$ attached along $L$ via the identitiy maps $\mathrm{id}:L\subset Y_i\to L\subset Y_j$.
Consider $f:(Y,L)\to (K,L)$, where $f|_{Y_i}=f_i$. Clearly, $f$ is a strict map.
Since $Y-L=\coprod (Y_i-L)$ we obtain
$\mathrm{dim}_G(Y-L)\leq n$. 
\\For part (b), we have $H_*(Y,L)=\widetilde H_*(Y/L)=\widetilde H_*(\vee (Y_i/L))=\oplus\widetilde H_*(Y_i/L)=\oplus H_*(Y_i,L)$ and the induced map on homology \\$f_*:H_*(Y,L)=\oplus H_*(Y_i,L)\to  H_*(K,L)$ is such that $f_*|_{H_*(Y_i,L)}={f_i}_*$.

\end{proof}

In what follows we show that for $S^{CE}(M)$ finite, there exists a space $X$ and a cell-like map $q:M\to X$ such that for any element $\gamma\in S^{CE}(M)$, $\gamma$ can be traced back to $H_{n+1}(M_q,M,\mathbb L)$ by the following sequence:
$$H_{n+1}(M_q,M;\mathbb L)\stackrel{i_*}\to H_{n+1}(P_2(M),M;\mathbb{L})\stackrel{\delta}\to S(M)
$$

\begin{prop}
Let $M^n$ be a closed connected topological manifold, $n\geq6$, $p$ odd and $\alpha_1, ...,\alpha_k\in H_*(P_2(M),M;\mathbb{L}\wedge M(p))$, where $\mathbb{L}\wedge M(p)$ is $\mathbb L$-theory with coefficients in $\mathbb Z_p$. Then there is a cell-like map $q:M\to X$ and elements $\hat\alpha_1,...\hat\alpha_k\in H_*(M_q,M;\mathbb{L}\wedge M(p))$ such that $i_*(\hat\alpha_i)=\alpha_i$. 
\end{prop}
\begin{proof}  Let $n\geq7$ and $\alpha_1, ...,\alpha_k\in H_*(P_2(M),M;\mathbb{L}\wedge M(p))$. Then there is a finite complex $K$, $M\subset K\subset P_2(M)$ and elements $\gamma_1,...\gamma_k\in H_*(K,M;\mathbb{L}\wedge M(p))$ such that $\gamma_i$ is taken to $\alpha_i$ by the inclusion homomorphism. By Theorem 1.1 $\widetilde H_*(K(\mathbb Z,3);\mathbb{L}\wedge M(p))=0$, hence it follows from Lemma 3.8 that there is a compactum $Y\supset M$ and a strict map $f:(Y,M)\longrightarrow (K,M)$ such that $\mathrm{dim}_\Z (Y-M)\leq 3$ and $\alpha_1,...,\alpha_k\in \mathrm{Im}(f_*).$ The rest of the proof and the case $n=6$ is exactly like the proof of Proposition 3.6 in \cite{Dr1}.
\end{proof}

\begin{thm} Let $M^n$ be a closed connected topological manifold, $n\geq6$. If $\beta_1,...,\beta_k\in H_*(P_2(M), M;\mathbb{L})$ are odd torsion elements of order$(\beta_i)=p_i$, then there exist a cell-like map $q:M\to X$ 
and elements $\hat\beta_1,...,\hat\beta_k\in H_*(M_q,M,;\mathbb{L})$ such that $i_*(\hat\beta_i)=\beta_i$. 
\end{thm}
\begin{proof} Let $l=p_1p_2\cdot\cdot\cdot p_k$. Consider the following commutative diagram of universal coefficient formulas
$$
{\xymatrix@C.0in{H_{*+1}(M_q,M;\mathbb{L}\wedge M(l))\ar[r]^{\phi'}\ar[d]^{i_*}&  \mathrm{Tor}(H_*(M_q,M;\mathbb{L}),\,\Z_l)
\ar[r]^{\ \ \subset}\ar[d] & H_*(M_q,M;\mathbb{L})\ar[d]^{i_*}\\
H_{*+1}(P_2(M),M;\mathbb{L}\wedge M(l))\ar[r]^{\phi} &  \mathrm{Tor}(H_*(P_2(M),M;\mathbb{L}),\,\Z_l)\ar[r]^{\ \ \ \ \ \subset} &
H_*(P_2(M),M;\mathbb{L})\\
}}
$$
where $\phi$ and $\phi'$ are epimorphisms. Let $\beta_1,...,\beta_k\in H_*(P_2(M), M;\mathbb{L})$, thus  $\beta_1,...,\beta_k\in\mathrm{Tor}(H_k(P_2(M),M;\mathbb{L}),\,\Z_l)$. Pick $\alpha_i\in\phi^{-1}(\beta_i)$ for each $i$. By Proposition 3.9, there is a cell-like map $q:M\to X$ and $\hat\alpha_1,...\hat\alpha_k\in H_{*+1}(M_q,M;\mathbb{L}\wedge M(l))$ such that $i_*(\hat\alpha_i)=\alpha_i$. It follows from the commutativity of the diagram that $i_*(\hat\beta_i)=\beta_i$ where $\hat\beta_i=\phi'(\hat\alpha_i)$.
\end{proof}  

\begin{defn} Let $f:(Y,L)\to(X,L)$ be a strict map. A homotopy $f_t:Y\to X$ which is strict at each level is called $\emph{strict}$ if the homotopy $f_t:(Y,L)\to(X,L)$ is continuous.
\end{defn} 

\begin{defn}\cite{Dr1}
    Let $Y$ be an open manifold and let $\bar Y$ be a compactification of an end of $Y$ by $X$, $Y=\bar Y-X$ . A $\emph{strict homotopy equivalence}$ \\${near}\ X$ is
        a strict map $\bar f:(\bar W,X)\to (\bar Y,X)$, where $\bar W$ is a compactification of an end of $W$ by $X$, such that there are neighborhoods $\bar U \supset \bar V$ of $X$ in $\bar W$ and $\bar U' \supset \bar V'$ of $X$ in $\bar Y$ such that $\bar f(\bar U) \subset \bar U'$ and there is a strict map $\bar g:(\bar U',X) \to (\bar U,X)$ such that
        \begin{enumerate}
        \item [(a)] $\bar g \circ \bar f|_{\bar V}$ is strict homotopic in $\bar U$ to $id_{\bar V}.$
        \item [(b)] $\bar f \circ \bar g |_{V'}$ is strict homotopic in $\bar U'$ to $id_{\bar V'}$.
        \end{enumerate}
        \end{defn}

\begin{defn}\cite{Dr1} The set of {\it germs of continuously controlled structures} on $Y$ at $X$, denoted $S^{cc}(\bar Y,X)_{\infty}$, is the set of equivalence classes of strict homotopy equivalences of manifolds near $X$. Two strict homotopy
        equivalences near $X$, $\bar f:(\bar W,X)\to (\bar Y,X)$
        and $\bar f':(\bar W',X)\to (\bar Y,X)$ are
        {\it equivalent} if there exist a neighborhood $\bar V$ of $X$ in
        $\bar W$
        and a strict map $\bar h:(\bar V,X) \to (\bar W',X)$ which is
        an open  imbedding
        and  $\bar f'\circ \bar h:\bar V\to \bar Y$ is  strict homotopic to
        $\bar f|_{\bar V}$.
\end{defn}

\noindent Let $q:M\to X$ be a cell-like map of a closed connected manifold $M$ and $M_q$ its mapping cylinder. Let $\stackrel\circ M_q=M_q-M\times\{0\}$.
 
 \begin{rem}(refer to \cite{Dr1}, Corollary 4.6 and section 5): Let $[\bar g]$ be an element of $S^{cc}(\stackrel\circ M_q,X)_{\infty}$. Then $\bar g:(\bar W,X)\to (\stackrel\circ M_q,X)$ is a strict homotopy equivalence near $X$. We can assume that $W=N\times(0,1)$ and $\bar W=\stackrel\circ M_p$, where $p:N\to X$ is cell-like. The forgetful map $\phi:S^{cc}(\stackrel\circ M_q,X)_{\infty}\to S(M)$ sends $[\bar g]$ to $[f]$, where $f:N\to M$ is a homotopy equivalence that factors through the cell-like maps $p$ and $q$.
\end{rem}

\begin{prop}(\cite{Dr1}, Proposition 5.4) Let $M$ be a closed connected $n$-manifold and $q:M\to X$ a cell-like map. Then the forgetful map $\phi:S^{cc}(\stackrel\circ M_q,X)_{\infty}\to S(M)$ factors as 
$$
S^{cc}(\stackrel\circ M_q,X)_{\infty}\stackrel{j}\to
H_{n+1}(M_q,M;\mathbb{L})\stackrel{i_*}\to
H_{n+1}( P_2(M),M;\mathbb{L})\stackrel{\delta}\to S(M)
$$
where $j$ is a monomorphism with cokernel $\mathbb{Z}$ or $0$.
\end{prop}

\begin{thm}
Let $M$ be an $n$-dimensional closed connected topological manifold, $n\ge6$ and $G$ be a finitely generated subgroup of  $S^{CE}(M)$. Then there exists a space $X$ such that for any $[N,f]\in G$, $f:N\to M$ factors through cell-like maps to $X$.
\end{thm}

\begin{proof} It follows from Theorem 2.10 that $G$ is finite. Let $\gamma_1,...,\gamma_k\in G$. Then there are odd torsion elements $\gamma'_1,...\gamma'_k$ in $H_{n+1}(P_2(M),M;\mathbb{L})$ such that $\partial(\gamma_i)=\gamma'_i$. By Theorem 3.12, there exist a cell-like map $q:M\to X$ and elements $\beta_1,...,\beta_k\in H_*(M_q,M,;\mathbb{L})$ such that $i_*(\beta_i)=\gamma'_i$. The result follows from Proposition 3.17.
\end{proof}

\begin{thm}
For $n\geq 6$ and any $m$, there are $m$ closed non-homeomorphic $n$-manifolds that can be mapped by cell-like maps onto the same space $X$.
\end{thm}
\begin{proof}
Let $p$ and $M$ be as in Corollary 3.4. Then, by Corollary 3.6, any two classes in $S^{CE}(M)$ contain non-homeomorphic manifolds. The result follows from Theorem 3.18.
\end{proof}

\end{document}